\pdfoutput=1
\documentclass[reqno,11pt,final]{amsart}

\usepackage[english]{babel}

\usepackage[T1]{fontenc}
\usepackage[utf8]{inputenc}

\usepackage[final, 	babel 	]{microtype}

\usepackage{mathtools} 
\usepackage{csquotes} \usepackage[backend=biber,
	style=numeric, 	firstinits=true,
	hyperref=auto, 	isbn=false,
	urldate=comp
	]{biblatex}
\usepackage{hyperref}

\numberwithin{equation}{section}

\newcommand\goth{\mathfrak}

\newcommand\cc{\mathbb{C}}

\newcommand\rr{\mathbb{R}}
\newcommand\zz{\mathbb{Z}}

\newcommand\bbS{\mathbb{S}}
\newcommand\GG{\mathbb{G}}

\newcommand\hhh{\mathbb{H}}

\newcommand\da{{\partial}_{\!\mbox{\tiny A}}}

\newcommand\ddd{{/\!\!\!\partial}}

\newcommand\dda{{/\!\!\!\partial}_{\!\!\mbox{\tiny A}}}

\newcommand\ot{\otimes}
\newcommand\op{\oplus}

\DeclareMathOperator{\Div}{div} \DeclareMathOperator{\grad}{grad}  \DeclareMathOperator{\Spin}{Spin} \DeclareMathOperator{\SO}{SO} 
\newcommand{\ACI}{\mathfrak{I}} \newcommand{\ic}{i}

\newtheorem{lemma}{Lemma}[section]

\newtheorem{proposition}{Proposition}[section]

\begin{document}

\title[]{The functional of super Riemann surfaces -- a ``semi-classical'' survey}
\author{Enno Keßler \and Jürgen Tolksdorf}
\address{Max Planck Institute for Mathematics in the Sciences, Leipzig, Germany}
\email{kessler@mis.mpg.de}
\email{Juergen.Tolksdorf@mis.mpg.de}

\thanks{The research leading to these results has received funding from the European Research Council under the European Union's Seventh Framework Program (FP7/2007--2013) / ERC grant agreement~nº267087.}

\dedicatory{In honor of Prof.\ E.\ Zeidler’s 75th birthday}

\begin{abstract}
	This article provides a brief discussion of the functional of super Riemann surfaces from the point of view of classical (i.e.\ not ``super-'') differential geometry.
	The discussion is based on symmetry considerations and aims to clarify the ``borderline'' between classical and super differential geometry with respect to the distinguished functional that generalizes the action of harmonic maps and is expected to play a basic role in the discussion of ``super Teichm\"uller space''.
	The discussion is also motivated by the fact that a geometrical understanding of the functional of super Riemann surfaces from the point of view of super geometry seems to provide serious issues to treat the functional analytically.
\end{abstract}

\maketitle
\tableofcontents

\noindent
{\bf PACS Classification:}
02.40.Ky, 03.50.Kk, 03.65.Sq, 11.30.Pb, 12.60.Jv

\noindent
{\bf MSC Classification:}
58A50, 14H55, 32G15\\[.08cm]

\noindent
{\bf Keywords:} Clifford Modules, Dirac Operators, Torsion, Non-linear $\sigma-$models, Super Riemann
surfaces, Supersymmetry.

\section{Introduction}
In this section we give a brief account on the functional of ``super harmonic maps'', which will be discussed afterwards in some detail from a classical differential geometrical setting in view of Dirac operators and symmetry considerations.
Accordingly, in this section we are intentionally sketchy with the aim to only provide some motivation for what is following without cluttering the brief outline with too many technical details.
A more rigorous and clear exposition, especially with respect to the mappings and geometrical constructions, as well as the notation used in this section, will be postponed to subsequent sections.

To begin with let $(M_k,g_k)$ be Riemannian manifolds $(k=1,2)$ where $M_1$ is assumed to be closed compact and orientable. Also, let $\varphi\colon M_1\rightarrow M_2$ be a smooth mapping.

The \emph{functional of harmonic maps}
\begin{equation}
\label{ahm}
	\mathcal{I}_{\mbox{\tiny H}}(\varphi) \coloneqq \int_{M_1}\|d\varphi\|^2\,dvol(g_1)
\end{equation}
is known to play a basic role in geometric analysis.
Here,
\begin{equation}
	\|\alpha\ot w\|^2(t) \coloneqq g_1^\ast|_t(\alpha,\alpha)\,g_2|_{\varphi(t)}(w,w)
\end{equation}
for all $\alpha\ot w\in\Omega^1(M_1,\varphi^\ast TM_2)$, where $g_1^\ast$ denotes the metric on the co-tangent bundle of $M_1$.
The study of the functional~\eqref{ahm} has a long history, actually (e.g.\ see~\cite{ES-HMRM}, \cite{JJ-RGGA} and references therein).

In fact, for $M_1 \coloneqq [0,1]\subset\rr$, the minimizers of~\eqref{ahm} are but the geodesics of $(M_2,g_2)$.
Furthermore, the functional of harmonic maps plays a basic role in the analytical treatment of minimal surfaces and generalizations thereof.

A particularity of the case where $M_1$ is two-dimensional, is that the functional~\eqref{ahm} is not only diffeomorphism invariant but also conformally invariant.
This crucial feature allows to apply the action functional~\eqref{ahm} to the study of Teichmüller space, which is a contractible covering of the moduli space of compact one dimensional complex manifolds.
For instance, for harmonic maps between surfaces the ``energy-momentum tensor'', defined by the variation of~\eqref{ahm} with respect to $g_1$, can be geometrically interpreted as a tangent vector to the Teichm\"uller space at the point \(M_1\).
Moreover, it can be shown using~\eqref{ahm} that the tangent space to the Teichmüller space at \(M_1\) can be identified with holomorphic quadratic differentials on \(M_1\).

The functional~\eqref{ahm} has then been extended, for instance, by including spinors, to the \emph{functional of Dirac harmonic maps} (c.f.~\cite{CJLW-DHM},~\cite{CJLW-RTEIDHM})
\begin{equation}
\label{adhm}
	\mathcal{I}_{\mbox{\tiny DH}}(\varphi,\psi) \coloneqq \int_{M_1}\Big(\|d\varphi\|^2 + \big<\psi,\ddd\psi\big>_{\!\mbox{\tiny$\mathcal{E}$}}\Big) dvol(g_1)\,.
\end{equation}
Here the field \(\psi\) is a section of the twisted spinor bundle $\pi_{\mbox{\tiny$\mathcal{E}$}}\colon \mathcal{E}\coloneqq\bbS\!\ot\!\varphi^\ast TM_2\rightarrow M_1$.
Consequently, one has to assume that $M_1$ is a spin manifold and there is a fiber metric $\big<\cdot,\cdot\big>_{\!\mbox{\tiny$\bbS$}}$ on a corresponding spinor bundle $\pi_{\mbox{\tiny$\bbS$}}\colon \bbS\rightarrow M_1$, such that the Dirac operator is symmetric.
The Dirac operator $\ddd$ is the ``quantized'' Clifford connection $\nabla^{\mbox{\tiny$\mathcal{E}$}}$ on the twisted spinor bundle \(\mathcal{E}\) that arises from the Levi-Civita connections on $(M_1,g_1)$ and $(M_2,g_2)$, respectively.
Also, the fiber metric reads $\big<\cdot,\cdot\big>_{\!\mbox{\tiny$\mathcal{E}$}} \equiv \big<\cdot,\cdot\big>_{\!\mbox{\tiny$\bbS$}}\ot\varphi^\ast g_2$.

In the study of Dirac harmonic maps it is decisive to use representations of the Clifford algebra \(Cl_{0,2}\) to construct the Dirac operator.
Otherwise, the real-valued \emph{``Dirac action''}
\begin{equation}
\label{da}
	\big<\!\!\big<\psi,\ddd\psi\big>\!\!\big> \coloneqq \int_{M_1}\big<\psi,\ddd\psi\big>_{\!\mbox{\tiny$\mathcal{E}$}}\, dvol(g_1)
\end{equation}
vanishes in the case where \(\dim M_1=2\).
Contrary to what is custom in the study of Dirac harmonic maps, we will work with representations of \(Cl_{2,0}\) and look at other ways to prevent the vanishing of the Dirac action~\eqref{da}.
The major challenge is that the Clifford algebra $Cl_{2,0}$ of the Euclidean plane $\rr^{2,0}$ has no skew-symmetric representation on the underlying \emph{real} spinor module.
In the physics real spinors are called \emph{Majorana spinors}.
We adopt this terminology.
The consequences of the vanishing of~\eqref{da} in case of Majorana spinors will be explained in some detail in what follows.
Especially, we shall discuss the usual way out of this flaw by the assumption of \emph{odd} (``anti-commuting'') spinors (see below).
Indeed, the functional~\eqref{da} becomes non-trivial even for real anti-commuting Majorana spinors on Riemann surfaces.
Furthermore, incorporating anti-commuting spinors, the action functional~\eqref{adhm} is not only invariant under diffeomorphims and conformal re-scaling of the metric $g_1$ (i.e. ``Weyl transformations'') but also with respect to the variations of~\eqref{adhm} that are defined to first order by so-called \emph{supersymmetry transformations}
\begin{equation}
\label{susytraf}
	\begin{split}
		\delta_{\goth s}\varphi &= \big<{\goth s},\psi\big>_{\!\mbox{\tiny$\bbS$}}\in\Gamma(M_1,\varphi^\ast TM_2)\,,\\
		\delta_{\goth s}\psi &= \gamma(d\varphi){\goth s}\in\Gamma(M_1,\mathcal{E})\,.
	\end{split}
\end{equation}
Here, the variational spinor fields ${\goth s}\in\Gamma_{\!\mbox{\tiny hol}}(M,\bbS)$ are restricted to arbitrary (odd) \emph{holomorphic} sections.

To avoid this restriction will bring us eventually to the even more enhanced (real-valued) \emph{functional of super harmonic maps}
\begin{gather}
\label{sadhm1}
	\mathcal{I}_{\mbox{\tiny SDH}}(\varphi,\psi,\chi,g_1) \coloneqq\\\nonumber
	\!\int_{M_1}\!\!\Big(\|d\varphi\|^2 + \big<\psi,\ddd\psi\big>_{\!\mbox{\tiny$\mathcal{E}$}}
		+ \big<\chi,{\goth q}(\chi)\big>_{\mbox{\tiny$T^\ast\!M_1\!\ot\!\bbS$}}\big<\psi,\psi\big>_{\!\mbox{\tiny$\mathcal{E}$}}
	+ 4\,\big<{\goth q}(\chi)(grad\varphi),\psi\big>_{\!\mbox{\tiny$\mathcal{E}$}}\Big)dvol(g_1)\,,
\end{gather}
with the section $\chi\in\Omega^1(M_1,\bbS)$ being termed \emph{gravitino}.
Again, the notation used will be explained in more detail in the next section.
Notice also that the functional of super harmonic maps is older than the functional of Dirac harmonic maps as it has been studied already in the seventies in the context of non-linear super symmetric sigma models and string theory (see~\cite{BdVH-LSRIASS}, \cite{DZ-CASS}).

Besides the diffeomorphism and conformal invariance (for $\dim(M_1) = 2)$, the functional~\eqref{sadhm1} carries an additional symmetry.
In fact, it is also invariant with respect to \emph{super Weyl transformations}
\begin{equation}
\label{swtraf}
	\chi(v) \mapsto \chi(v) + \gamma(v^\flat){\goth s}\,,
\end{equation}
for all tangent vectors $v\in TM_1$ and arbitrary spinor field ${\goth s}\in\Gamma(M_1,\bbS)$.

In the \emph{two-dimensional case} and for \emph{odd spinors}, the functional~\eqref{sadhm1} has the crucial property that it does not depend on the metric connection on \(M_1\) appearing in the Clifford connection $\nabla^{\mbox{\tiny$\mathcal{E}$}}$.
One may therefore replace $\ddd$ by the Dirac operator $\dda$ with \emph{torsion}.
As a consequence, the functional~\eqref{sadhm1} becomes invariant also with respect to the \emph{enhanced supersymmetry transformations}
\begin{align}
\label{susytraf1a}
	\delta_{\goth s}\varphi &\coloneqq \big<{\goth s},\psi\big>_{\!\mbox{\tiny$\bbS$}}\in \Gamma(M_1,\varphi^\ast TM_2)\,,\\
\label{susytraf1b}
	\quad \delta_{\goth s}\psi &\coloneqq \gamma\big(d\varphi - \big<\psi,\chi\big>_{\!\mbox{\tiny$\bbS$}}\big){\goth s}\in \Gamma(M_1,\mathcal{E})\,,\\
\label{susytraf1c}
	\delta_{\goth s} e_k &\coloneqq -4\big<\delta_{\mbox{\tiny$\Theta^\sharp$}}{\goth s},\chi(e_k)\big>_{\!\mbox{\tiny$\bbS$}}\in \Gamma(M_1,TM_1)\qquad(k=1,2)\,,\\
\label{susytraf1d}
	\delta_{\goth s}\chi &\coloneqq d_{\mbox{\tiny A}}{\goth s}\in \Omega^1(M_1,\bbS)\,,
\end{align}
provided $g_2$ is flat and the torsion factorizes by the gravitino as
\begin{equation}
\label{torfac}
	\begin{split}
		A &= \big<\gamma(\chi),\chi\big>_{\!\mbox{\tiny$\bbS$}}\in\Omega^1(M_1)\,.
	\end{split}
\end{equation}
Here, $e_1,e_2\in\Gamma(M_1,TM_1)$ denotes an arbitrary (oriented) local $g_1-$orthonormal frame.

Notice that in contrast to~\eqref{susytraf}, the supersymmetry transformations~\eqref{susytraf1a}--\eqref{susytraf1d} are generated by a completely arbitrary (odd) variational spinor field ${\goth s}\in\Gamma(M_1,\bbS)$.
The somewhat simplifying assumption of $g_2$ being flat can be omitted, actually, by adding an appropriate curvature term to the integrand of~\eqref{sadhm1}.
In what follows, however, we restrict ourself to the case where $(M_2,g_2)$ is flat to keep things more straightforward.

What is the geometrical meaning of the somewhat ad hoc looking functional~\eqref{sadhm1} and how is it related to the functional~\eqref{ahm} of Dirac harmonic maps?
For more than 30 years there is the conjecture that the functional of super harmonic mappings~\eqref{sadhm1} is related to the moduli space of \emph{super Riemann surfaces} similar to how the functional of harmonic maps~\eqref{ahm} is related to the moduli space of Riemann surfaces (see however~\cite{dHP-GSP}).

Recall that a super Riemann surface is a complex super manifold $\mathcal{M}$ of dimension $1|1$ together with a rank $0|1$ dimensional distribution $\mathcal{D}\subset\mathcal{TM}$, such that $\mathcal{D}\ot\mathcal{D}\simeq_\cc\mathcal{TM}/\mathcal{D}$.
Furthermore, in~\cite{EK-DR} it has been shown how the functional~\eqref{sadhm1} of super harmonic maps can be re-written on super Riemann surfaces as
\begin{equation}
\label{sahm}
	{\mathcal I}_{\mbox{\tiny SDH}}(\Phi) = \int_{\mathcal{M}}\|d\Phi|_{\mathcal{D}}\|^2[dvol]\,,
\end{equation}
where $\Phi\colon \mathcal{M}\rightarrow\mathcal{N}$ is a mapping between (families of) super manifolds which is the analogue of $\varphi\colon M_1\rightarrow M_2$ in the case of~\eqref{ahm}.
The formal similarity between~\eqref{sahm} and~\eqref{ahm} is apparent.
Within the super setting the supersymmetry transformations~\eqref{susytraf1a}--\eqref{susytraf1d} have the geometrical meaning of a specific (infinitesimal) diffeomorphism on the Riemann surface $\mathcal{M}$.
For details we refer to~\cite{JKT-SRSMG} and, especially, to~\cite{EK-DR}.
There, it is shown, for the first time, in all details that the functional~\eqref{sahm} is indeed well-defined on the moduli space of super Riemann surfaces like~\eqref{ahm} is well-defined on Teichm\"uller space.
Moreover, the geometrical meaning of the gravitino is shown to be related to the embedding of an underlying Riemann surface into $\mathcal{M}$.
Finally, the variation of~\eqref{sadhm1} with respect to the gravitino is demonstrated to have the geometrical meaning of a tangent vector field on the moduli space of super Riemann surfaces in ``odd'' directions, similar to the energy-momentum tensor is known to be tangent to the ``even'' directions.

In a way the geometrical meaning of the functional~\eqref{sadhm1} has been fully clarified in terms of super differential geometry.
Yet, one may pose the question to what extend the functional~\eqref{sadhm1} and super Riemann surfaces can also be motivated within the setting of ``classical'' (i.e.\ non-super) differential geometry.
In fact, in~\cite{JKT-SRSMG} it is discussed how super Riemann surfaces are related to spinor bundles over Riemann surfaces together with the gravitino field $\chi$.
This classical geometrical background seems more suitable for geometrical analysis than the super setting.
One reason for this might be the fact that within the super setting the functional~\eqref{sadhm1} is not longer real-valued but has to be understood as a mapping between super manifolds.
Actually, this seems to be unavoidable when one insists on supersymmetry also within the classical frame as it is exposed below.

In this work we review on how much of the classical setup can be retained to understand the functional of super harmonic maps~\eqref{sadhm1} together with the super Weyl and supersymmetry transformations~\eqref{swtraf}--\eqref{susytraf1d}.
Although much is known on this matter, already, it still seems worth presenting a detailed account on how supersymmetry enforces super geometry.
In fact, our discussion should be understood as being complementary to what has been presented in~\cite{JKT-SRSMG} and, in particular, in~\cite{EK-DR}.

\section{The geometrical setup}
In this section we summarize the geometrical background and fix the notation already used in the previous section.
The assumption that $M_1$ is two-dimensional and orientable has far reaching consequences.

Let $(M_1,g_1)$ be a two-dimensional orientable Riemannian manifold (Riemann surface).
The induced Riemannian volume form $\omega_1\equiv dvol(g_1)\in\Omega^2(M_1)$ yields an almost complex structure ${\goth I}\in{\rm End}(TM_1)$ via
\begin{equation}
\label{eq:DefAlmostComplexStructure}
	g_1(\mathfrak{I}u, v) \coloneqq \omega_1(u,v)\qquad(u,v\in TM_1)\,.
\end{equation}
Of course, this the almost complex structure would be the same for the metric rescaled by a positive function.
In other words, there is a one-to-one correspondence between almost complex structures and conformal classes of metrics on $M_1$.
This well-known fact can also been inferred from the fact that both almost complex structures and conformal classes of metrics on $M_1$ yield the same reduction of the structure group of the frame bundle by $\rr_+\times SO(2)\subset GL(2,\rr)$.

It is particular to the two dimensional case that any such almost complex structure is integrable.
Consequently, $M_1$ may be regarded as a one-dimensional complex manifold.
It follows that the complexified tangent bundle splits into the holomorphic and anti-holomorphic vector fields on $M_1$.
That is,
\begin{equation}
	TM_1^{\mbox{\tiny$\cc$}}\equiv TM_1\ot\cc \simeq T^{\mbox{\tiny(1,0)}}M_1\op T^{\mbox{\tiny(0,1)}} M_1\,.
\end{equation}
Clearly, the realification of $T^{\mbox{\tiny(1,0)}} M_1$ is but $TM_1$.

As \(\Spin(2)\simeq \SO(2)\), the real rank two \emph{Majorana spinor bundle} $\pi_{\mbox{\tiny$\bbS$}}\colon \bbS\rightarrow M_1$ associated to a chosen spin structure on \(M_1\) is also equipped with a symmetric scalar product and an antisymmetric bilinear form.
Those are given by the lifts of the metric \(g_1\) and the volume form \(\omega_1\) on \(TM_1\) and will be denoted by the same symbol on \(\bbS\).
Consequently, similar to Equation~\eqref{eq:DefAlmostComplexStructure}, there is an almost complex structure \(\ACI\) on \(\bbS\), that depends only on the conformal class of \(g_1\).

The complexification \(S=\bbS\otimes\cc\), is called the complex \emph{bundle of Dirac spinors} and splits into the eigenspaces of \(\ACI\) of eigenvalue \(\pm\ic\), such that
\begin{equation}
\label{paritydec}
	S\equiv\bbS\ot\cc = W\op\overline{W}.
\end{equation}
The mutually complex conjugate sub-vector bundles $W$ (eigenvalue \(+\ic\)) and $\overline{W}$ (eigenvalue \(-\ic\)) are called \emph{Weyl spinor bundles}.
The complex line bundle \(W\) is isomorphic as a complex line bundle to \(\bbS\), where the complex structure on the latter is given by the almost complex structure \(\ACI\).
Furthermore, as the chosen spin structure consists of a fiber wise double cover \(\Spin(2)\to\SO(2)\), one obtains the following identities of complex line bundles:
\begin{align}
\label{spinfac}
	W\ot W &\simeq T^{\mbox{\tiny(1,0)}} M_1\,,\\
\label{antispinfac}
	\overline{W}\ot\overline{W} &\simeq T^{\mbox{\tiny(0,1)}} M_1\,.
\end{align}
In fact, the choice of a line bundle \(W\) with the property~\eqref{spinfac} is equivalent to the choice of a spin structure.
A corollary of Equation~\eqref{spinfac} is that \(W\) is a holomorphic line bundle on the complex manifold \(M_1\).

Clearly, the factorization of the holomorphic vector fields on a Riemann surface into spinors~\eqref{spinfac} corresponds to the basic property of a super Riemann surface.
That is, the choice of a $0|1-$distribution $\mathcal{D}\subset\mathcal{M}$ in the ``super world'' reduces on the ``classical'' side to the choice of a spin structure on $M_1$.
Indeed, one can show, that trivial families of super Riemann surfaces are in one-to-one correspondence to Riemann surfaces with chosen spin structure, see e.g.~\cite[Proposition 4.2.2.]{S-GAASTS}.
To also capture non-trivial families of super Riemann surfaces, one needs the gravitino field \(\chi\) that appeared already in the first section.
We will now turn to the study of differential forms with spinor values, the bundle of which \(\chi\) is a section.

The twisted Dirac spinor bundle $\pi_{\mbox{\tiny$S\!\ot\!T^\ast\!M_1$}}\colon S\!\ot\!T^\ast\!M_1^{\mbox{\tiny$\cc$}}\rightarrow M_1$ decomposes as
\begin{equation}
\label{eq:DirectSumDecompositionOfSpinorValuedDifferentialForms}
	\begin{split}
		S\ot T^\ast\!M_1^{\mbox{\tiny$\cc$}} &\simeq S\oplus\big(W^3\op\overline{W}^3\big)\\
			&= \big(\bbS\op\GG\big)\ot\cc\,.
	\end{split}
\end{equation}
Here, the real sub-bundle $\pi_{\mbox{\tiny$\bbS\!\ot\!\GG$}}\colon \bbS\ot\GG\rightarrow M_1$ refers to the canonical real structure on $S\ot T^\ast\!M_1^{\mbox{\tiny$\cc$}}$.
Notice that the rank two complex vector bundle with total space $W^3\op\overline{W}^3$ has a canonical real structure $\GG$ in contrast to the complex line bundle with total space $W$ (or $\overline{W}$).
The latter has a real structure if and only if it is trivial.
Notice that \(\GG\) is the realification of the complex line bundle \(W^3\equiv W\otimes W\otimes W\).

As an upshot also the real twisted spinor bundle $\pi_{\mbox{\tiny$\bbS\!\ot\!T^\ast M_1$}}\colon \bbS\ot T^\ast\!M_1\rightarrow M_1$ becomes $\zz_2-$graded
\begin{equation}
\label{rtwspindec}
	\bbS\ot T^\ast\!M_1 \simeq \bbS\op\GG\,.
\end{equation}
Explicitly, the corresponding projection operators read
\begin{equation}
	\begin{split}
		{\goth p}\colon \bbS\ot T^\ast\!M_1 &\longrightarrow \bbS\ot T^\ast\!M_1\\
		\sigma_k\ot e^k &\mapsto \frac{1}{2}g_1(e_i,e_j)\gamma(e^j)\gamma(e^k)\sigma_k\ot e^i\,,\\[.2cm]
		{\goth q}\colon \bbS\ot T^\ast\!M_1 &\longrightarrow \bbS\ot T^\ast\!M_1\\
		\sigma_k\ot e^k &\mapsto \frac{1}{2}g_1(e_i,e_j)\gamma(e^k)\gamma(e^j)\sigma_k\ot e^i\,.
	\end{split}
\end{equation}
Here, and in the sequel, we take advantage of Einstein's summation convention.
Also, $e_1,e_2$ is a local (oriented) frame on \(TM_1\) with dual frame denoted by $e^1,e^2$.
Finally,
\begin{equation}
\label{cliffmap}
	\begin{split}
		\gamma\colon T^\ast\!M_1 &\longrightarrow {\rm End}(\bbS)\\
		\alpha &\mapsto \gamma(\alpha)
	\end{split}
\end{equation}
denotes a \emph{Clifford map}.
We also make use of the common notation: $\gamma^k\equiv\gamma(e^k)$, whenever $e^1,e^2$ is an oriented orthonormal basis with respect to $g_1^\ast$.

In more abstract terms the complementary projection operators ${\goth p}$ and ${\goth q} = 1 - {\goth p}$ are given by the following two mappings:
\begin{equation}
\label{fundmaps}
	\begin{split}
		\delta_\gamma\colon \bbS\ot T^\ast\!M_1 &\longrightarrow \bbS\\
		\sigma_k\ot e^k &\mapsto \gamma^k\sigma_k\,,\\[.2cm]
		\delta_{\mbox{\tiny$\Theta$}}:\,\bbS &\longrightarrow \bbS\ot T^\ast\!M_1\\
		\sigma &\mapsto \frac{1}{2}\delta_{ij}\gamma^i\sigma\ot e^j\,.
	\end{split}
\end{equation}
Since $\delta_\gamma\circ\delta_{\mbox{\tiny$\Theta$}} = 1$, one may define ${\goth p} \coloneqq \delta_{\mbox{\tiny$\Theta$}}\circ\delta_\gamma$.
We call $\delta_\gamma$ the \emph{quantization map} and simply write $\delta_\gamma(\alpha)\equiv\gamma(\alpha)\in{\rm End}(\bbS)$ for all $\alpha\in T^\ast\!M_1$.

Notice hat the projection operators ${\goth p}$ and ${\goth q}$ are self-adjoint, such that the decomposition
\begin{equation}
\label{orthodec}
	\bbS\ot T^\ast\!M_1 = {\goth p}(\bbS\ot T^\ast\!M_1)\op{\goth q}(\bbS\ot T^\ast\!M_1)
\end{equation}
becomes orthogonal.

Let $\phantom{x}^{\sharp/\flat}\colon T^\ast\!M_1\simeq TM_1$ be the ``musical'' isomorphisms defined by $g_1$
and its dual $g_1^\ast$, such that, for instance, $g^\ast_1(\alpha,\beta) = g_1(\alpha^\sharp,\beta^\sharp)$ for all
$\alpha,\beta\in T^\ast\!M_1$. We define for all ${\goth s}\in\bbS$ the canonical inclusion $\bbS\hookrightarrow\bbS\ot TM_1$
by $\delta_{\mbox{\tiny$\Theta^\sharp$}}{\goth s} \coloneqq \frac{1}{2}\gamma^k{\goth s}\ot e_k\in\bbS\ot TM_1$. Notice that
every Clifford map~\eqref{cliffmap} induces a canonical one-form $\Theta\in\Omega^1(M_1,{\rm End}(\bbS))$ that is given
by $\Theta(v) \coloneqq \frac{1}{2}\gamma(v^\flat)$ for all $v\in TM_1$. Explicitly,
$\Theta = \frac{1}{2}\delta_{ij}\gamma^i\ot e^j$. Accordingly, we put $\Theta^\sharp \coloneqq \frac{1}{2}\gamma^k\ot e_k$.

Any \(g_1\)-orthonormal frame \(e_1, e_2\) for \(TM_1\) gives rise to the hermitian frames \(e=\left(e_1-\ic e_2\right)/\sqrt{2}\) of \(T^{(1,0)}M_1\) and \(\overline{e}=\left(e_1+\ic e_2\right)/\sqrt{2}\) for \(T^{(0,1)}M_1\).
Similarly on \(\bbS\), a \(g_1\)-orthonormal frame \(\mathfrak{s}_1, \mathfrak{s}_2\) gives rise to a hermitian frame \(\mathfrak{w}=\left(\mathfrak{s}_1-\ic\mathfrak{s}_2\right)/\sqrt{2}\) on \(W\) and \(\overline{\mathfrak{w}}=\left(\mathfrak{s}_1 + \ic\mathfrak{s}_2\right)/\sqrt{2}\) on \(\overline{W}\).
We suppose furthermore that the frame for \(S\) covers the frame for \(TM_1\), i.e.\ \(\mathfrak{w}\otimes\mathfrak{w}=e\).
For the dual spaces \({\left(T^{(1,0)}M_1\right)}^*\) and \({\left(T^{(0,1)}M_1\right)}^*\) we use the dual basis of \(e, \overline{e}\), denoted by \(\theta\) and \(\overline{\theta}\) respectively.
Then, by construction of Equation~\eqref{eq:DirectSumDecompositionOfSpinorValuedDifferentialForms} it holds that
\begin{equation}
\label{eq:PQProjectors}
	\begin{split}
		{\goth p}(\bbS\ot T^\ast\!M_1) &= \left\{z\,{\goth w}\ot\theta + {\bar z}\,{\bar{\goth w}}\ot{\bar\theta}\,|\,z\in\cc\right\} \simeq\bbS\,,\\
		{\goth q}(\bbS\ot T^\ast\!M_1) &= \left\{z\,{\goth w}\ot{\bar\theta} + {\bar z}\,{\bar{\goth w}}\ot\theta\,|\,z\in\cc\right\}\simeq\GG\,.
	\end{split}
\end{equation}

The orthogonal decomposition~\eqref{rtwspindec} is but the irreducible decomposition of the twisted Majorana bundle $\bbS\ot T^\ast\!M_1$ into its spin-1/2 and a spin-3/2 parts.
That is, every ${\goth z}\in\bbS\ot T^\ast\!M_1$ has a unique decomposition
\begin{equation}
	{\goth z} = \delta_{\mbox{\tiny$\Theta$}}{\goth s} + {\goth g}\,,
\end{equation}
where ${\goth s}\in\bbS$ is uniquely determined by ${\goth s} \coloneqq \delta_\gamma({\goth z})$.
Likewise, the spin-3/2 spinor ${\goth g}\in\GG$ is uniquely determined by $\delta_\gamma({\goth g}) \coloneqq 0$.

It is amazing that the classical realm discussed so far can be basically subsumed by the fact that the Clifford algebra $Cl_{2,0}$ of the Euclidean plane $\rr^{2,0}$ decomposes as
\begin{equation}
\label{cliffdec}
	Cl_{2,0} \simeq \cc\op\rr^{2,0}
\end{equation}
and by the equality $Spin(2) = SO(2)$.
The latter identification allows to regard both the metric $g_1$ and the symplectic form $\omega_1\equiv dvol(g_1)$ as being inner products on $\bbS$.
That is, the notation $\big<\cdot,\cdot\big>_{\mbox{\tiny$\bbS$}}$ for the metric on the spinor bundle is but $g_1$.
The different notation used is just to indicate on whether $g_1$ acts as an inner product on spinors or on tangent vectors.

As mentioned already, in the usual approach to the action of Dirac harmonic maps~\eqref{adhm} one considers the Clifford algebra $Cl_{0,2}$, instead of $Cl_{2,0}$.
This is to avoid the flaw of a vanishing Dirac action.
Here, one takes into account that $Cl_{0,2}\simeq_\rr\hhh$ and identifies the latter, as a vector space, with $\cc^2$.
Notice that the only spinor module of $Cl_{0,2}$ is given by the Clifford algebra $Cl_{0,2}$ itself.
Hence, in the usual approach to Dirac harmonic maps on Riemann surfaces one identifies spinors with sections of a complex vector bundle of rank two.
Clearly, this spinor module carries a skew-hermitian representation of the Clifford action.
However, in this approach one loses much of the structure contained in the decomposition~\eqref{cliffdec}.
In particular, one loses the factorization~\eqref{spinfac}, which is at the very heart of the definition of super Riemann surfaces and the notion of gravitinos.
The meaning of the latter within the classical realm will be discussed next.
 
\section{Torsion on Riemann surfaces}
With a connection $\nabla$ on the tangent bundle of an arbitrary smooth manifold $M$ there are associated two different geometrical objects: the curvature and the \emph{torsion} of this connection.
The torsion may be defined as
\begin{equation}
	\tau_{\mbox{\tiny$\nabla$}} \coloneqq d_{\mbox{\tiny$\nabla$}}{\goth{Id}}\,,
\end{equation}
where, respectively, $d_{\mbox{\tiny$\nabla$}} $ and ${\goth{Id}}\in\Omega^1(M,TM)$ are the exterior covariant derivative with respect to the connection $\nabla$ and the canonical one-form that corresponds to the soldering form on the frame bundle of $M$.
On an $n-$dimensional orientable Riemannian manifold $(M,g)$ one may describe the torsion of a metric connection equivalently in terms of a one-form $A\in\Omega^1(M,so(n))$ via
\begin{equation}
	\tau_{\mbox{\tiny$\nabla$}}(u,v) \eqqcolon A(u)v - A(v)u\qquad(u,v\in TM)\,.
\end{equation}

In particular, in the case of a Riemann surface the torsion of the most general metric connection reads
\begin{equation}
	\tau_{\mbox{\tiny$\nabla$}} = A\ot{\goth I}\,,
\end{equation}
with $A\in\Omega^1(M_1)$ being an ordinary one-form.
Accordingly, the torsion can be lifted to the (real) spinor bundle $\bbS$ as $\frac{1}{2}A\ot\gamma^1\gamma^2\in\Omega^1(M_1,{\rm End}(\bbS))$.

Notice that the most general metric connection on $M_1$ reads
\begin{equation}
\label{vantor}
	\begin{split}
		\nabla^{\mbox{\tiny$g$}} &= \nabla^{\mbox{\tiny LC}} + A\ot{\goth I}\\
		&\stackrel{loc.}{=} d + (\Gamma + A)\ot{\goth I}\,,
	\end{split}
\end{equation}
where $\Gamma(e_k) \coloneqq g_{\mbox{\tiny M}}(\nabla^{\mbox{\tiny LC}}_{\!\!e_k}e_1,e_2)\;(k=1,2)$ is the connection form of the Levi-Civita connection with respect to an arbitrary (oriented) $g_1-$orthonormal basis.
We denote the induced connection on the (real) spinor bundle by $\da \coloneqq \nabla^{\mbox{\tiny S}} + \frac{1}{2}A\ot\gamma^1\gamma^2$, with $\nabla^{\mbox{\tiny S}} \stackrel{loc.}{=} d + \frac{1}{2}\Gamma\ot\gamma^1\gamma^2$ being the ordinary spin connection on $\bbS$.

In~\cite{T-EHACCFGF} it has been demonstrated how the functional
\begin{equation}
\label{dym}
	\int_{M_1}\!\!\Big(\big<\psi,\dda\psi\big>_{\mbox{\tiny$\bbS\!\ot\!T^\ast\!M_1$}} + \|F_{\!\mbox{\tiny A}}\|^2\Big)dvol(g_1)
\end{equation}
can be derived from a specific class of Dirac operators. Here, respectively, $\dda$ is the quantized connection $\da$ and $F_{\!\mbox{\tiny A}} \coloneqq dA$ is the curvature form associated to the torsion $\tau_{\mbox{\tiny$\nabla$}}$.

Even more, for smooth mappings $\varphi\colon M_1\rightarrow M_2$  between arbitrary Riemann manifolds it has been also shown in~\cite{T-EHACCFGF} how the functional of \emph{Dirac harmonic maps with torsion}
\begin{equation}
\label{dhym}
	\int_{M_1}\!\!\Big(scal(g_1) + \|d\varphi\|^2 + \big<\psi,\dda\psi\big>_{\mbox{\tiny$\bbS\!\ot\!T^\ast\!M_1$}} + \|F_{\!\mbox{\tiny A}}\|^2\Big)dvol(g_1)\,.
\end{equation}
naturally fits with the geometry of \emph{Dirac operators of simple type}.

However, in the two-dimensional case the part that involves the scalar curvature on $M_1$ becomes a topological invariant (Euler characteristic).
Up to this constant the functional~\eqref{dhym} is known to be Weyl invariant if and only if $F_{\!\mbox{\tiny A}} = 0$. In this case, one may always find a gauge (i.e.\ a frame), such that torsion vanishes, locally (c.f.~\eqref{vantor}).
Hence, by imposing conformal invariance~\eqref{dhym} reduces to the functional~\eqref{adhm} of Dirac harmonic maps.
This demonstrates how the latter naturally fits with the geometry of Dirac operators of simple type.
Let us point out, however, that by the geometrical construction indicated the (local) vanishing of torsion on Riemann surfaces is implemented by a co-homological condition instead put in by hand right-a-way.
This co-homological condition may be viewed as the analogue of the independency of the functional of super harmonic maps~\eqref{sadhm1} from the torsion, although torsion is around due to~\eqref{susytraf1d}.

How does the gravitino enter the classical stage?
The answer is provided by the factorization condition~\eqref{torfac} imposed on the torsion by the demand of supersymmetry?

Indeed, one has the following
\begin{proposition}
	On a Riemann surface every $A\in\Omega^1(M_1)$ factorizes by a section $\chi\in\Gamma(M_1,\bbS\!\ot\!T^\ast\!M_1)$,
	i.e.
	\begin{equation}\label{eq:GravitinoFactorizesTorsion}
		A = \big<\gamma(\chi),\chi\big>_{\mbox{\tiny$\bbS$}}\,.
	\end{equation}
\end{proposition}
\begin{proof}
	The proof basically takes advantage of the factorizations~\eqref{spinfac}--\eqref{antispinfac}.

	Explicitly, let \(A\) be given by $A = a_1e^1 + a_2e^2 = a\theta + {\bar a}{\bar\theta}\in\Omega^1(M_1)$ in terms of ${\goth s}\in{\goth S}ec(M_1,\bbS)$.
	We may define the spinor \(\mathfrak{s}\) and the gravitino \(\mathfrak{g}\) in terms of $A$ as follows
	\begin{equation}
		\begin{split}
			{\goth s} &\coloneqq {\rm Re}\sqrt{a}\,{\goth s}_1 - {\rm Im}\sqrt{a}\,{\goth s}_2 = a_1{\goth s}_1 + a_2{\goth s}_2\in\bbS\,,\\
			{\goth g} &\coloneqq {\rm Re}\frac{{\bar a}}{\sqrt{a}}\big({\goth s}_1\ot e^1 - {\goth s}_2\ot e^2\big) +
			{\rm Im}\frac{{\bar a}}{\sqrt{a}}\big({\goth s}_2\ot e^1 + {\goth s}_1\ot e^2\big)\\
			&=
			\frac{1}{\|A\|}\left({\rm Re}\,{\bar a}^{3/2}\big({\goth s}_1\ot e^1 - {\goth s}_2\ot e^2\big) +
			{\rm Im}\,{\bar a}^{3/2}\big({\goth s}_2\ot e^1 + {\goth s}_1\ot e^2\big)\right)\in\GG\,.
		\end{split}
	\end{equation}
	Apparently, the coefficients of the gravitino ${\goth g}$ transform with respect to a $3/2-$repre\-sentation of $Spin(2) = SO(2)$, whereas the coefficients of ${\goth s}$ with respect to the fundamental representation of the spin group.
	One calculates that
	\begin{equation}
		\langle{\goth s},{\goth g}\rangle_{\mbox{\tiny$\bbS$}} = A\,.
	\end{equation}

	Alternatively, one may write
	\begin{equation}
		\chi \coloneqq \frac{1}{\sqrt{2}}\big(\delta_{\mbox{\tiny$\Theta$}}{\goth s} + {\goth g}\big)\in\Omega^1(M_1,\bbS)\,,
	\end{equation}
	such that
	\begin{equation}
		A = \langle\delta_\gamma(\chi),\chi\rangle_{\mbox{\tiny$\bbS$}}\,.
	\end{equation}
\end{proof}

Clearly, given \(\chi\in \Omega^1(M_1, \bbS)\) the formula~\eqref{eq:GravitinoFactorizesTorsion} defines a torsion.
Yet, the correspondency between the torsion and the gravitino (i.e. $A\leftrightarrow\chi$) is far from being unique.
Of course, the assumption $F_{\!\mbox{\tiny A}} = 0$ (enforced by ordinary conformal symmetry) remedies this factorization ambiguity, for it guarantees that the functional~\eqref{dhym} does not depend at all on torsion.
Indeed there are many ways to factorize torsion in terms of the section \(\chi\).
One may thus pose the question, what is the invention of torsion good for looking at the functional of Dirac harmonic maps from the classical point of view? In the language of physics the answer can be expressed as follows.
The factorization of torsion by gravitino fields allows to introduce additional ``couplings'' (i.e.\ invariants) in terms of the gravitino which would not be appear otherwise.
Of course, these additional invariants should be compatible with all the symmetries imposed.
Especially, the couplings should obey conformal symmetry.
When further restricted to at most quadratic invariants one ends up with the following three conformal invariants:
\begin{align}
\label{susyint1}
	2\big<\psi,ev\big(\chi\ot{\goth q}(\chi)\big)\psi\big>_{\!\mbox{\tiny$\mathcal{E}$}} &= \big<\chi_i,\gamma^j\gamma^i\chi_j\big>_{\!\mbox{\tiny$\bbS$}} \big<\psi,\psi\big>_{\!\mbox{\tiny$\mathcal{E}$}}\,,\\
\label{susyint2}
	2\big<\psi,\big(\chi\ot{\goth q}(\chi)\big)\psi\big>_ {\!\mbox{\tiny$\mathcal{E}$}} &= \big<\chi_i,\gamma^j\gamma^i\psi^k\big>_{\!\mbox{\tiny$\bbS$}} \big<\psi_k,\chi_j\big>_{\!\mbox{\tiny$\bbS$}}\,,\\
\label{susyint3}
	2\big<{\goth q}(\chi)(\grad\varphi),\psi\big>_{\!\mbox{\tiny$\mathcal{E}$}} &= \delta^{ij}\big<\chi_i\ot\partial_j\varphi,\psi\big>_{\!\mbox{\tiny$\mathcal{E}$}}\,.
\end{align}

Super Weyl invariance eventually allows to reduce these couplings further to exactly the two coupling terms that appear in the functional of super harmonic maps~\eqref{sadhm1}.
Indeed, the conformal invariant~\eqref{susyint2} is ruled out by imposing invariance also under~\eqref{swtraf}.
We stress that the invariance of~\eqref{susyint1} under super Weyl transformations is guaranteed by the orthogonal decomposition~\eqref{orthodec}.

Having clarified how the functional of super harmonic mappings may be motivated within the geometrical setup of classical differential geometry, we now turn to the question how to motivate the supersymmetry transformations~\eqref{susytraf1a}--\eqref{susytraf1d} from this point of view.
This is mainly due to the triviality of the Dirac action for real spinors, as will be discussed next.
 
\section{The Dirac action on Riemannian surfaces and supersymmetry}
In the case of a (closed compact) two-dimensional Riemannian manifold the Dirac action~\eqref{da} vanishes in the case of Majorana (i.e.\ real) spinors.
The reason for this is that the Clifford algebra $Cl_{2,0}$ has no skew-symmetric Majorana representations.
In fact, the real spinor module only allows for symmetric representations of $Cl_{2,0}$.
Hence, the standard spin Dirac operator $\ddd$ is skew-symmetric when acting on Majorana fields.
Therefore, when switching to Dirac (i.e.\ complex) spinors, one may replace $\ddd$ by $i\ddd$ to obtain a hermitian Dirac operator.
However, in this case the functional of super harmonic maps is not guaranteed to be real, which makes an analytical treatment of this functional more complicated.

The usual ``way out'' of this dilemma is to consider spinors as being ``odd'' (or ``anti-commuting'') objects.
More precisely, one assumes that there exist “superized” extensions \(\hat{\bbS}\) of \(\bbS\), \(\hat{g}_1\) of the spinor metric \(g_1\), and \(\hat{\omega}_1\) of \(\omega_1\) such that for all spinors $\hat{\goth s},\hat{\goth s}'\in\hat{\bbS}$
\begin{equation}
\label{oddspinor}
	\begin{split}
		\hat{g}_1(\hat{\goth s}, \hat{\goth s}') = -\hat{g}_1(\hat{\goth s}', \hat{\goth s})\,,\\
		\hat{\omega}_1(\hat{\goth s}, \hat{\goth s}') = +\hat{\omega}_1(\hat{\goth s}', \hat{\goth s})\,.
	\end{split}
\end{equation}
Note that the signs in~\eqref{oddspinor} follow the rule that whenever two odd spinors are permuted an extra sign is acquired.

The main motivation for this sign rule is that it remedies the main cause for the vanishing of the Dirac action, namely the non-existence of skew-symmetric representations for \(Cl_{2,0}\).
It holds that every $g_1-$symmetric Majorana representation of $Cl_{2,0}$ is $\hat{\omega}_1-$skew-symmetric and vice versa:
\begin{equation}
\label{symrep}
	g_1({\goth s}, \gamma(\alpha){\goth s}') =
	g_1(\gamma(\alpha){\goth s}, {\goth s}')  \quad\Leftrightarrow\quad
		\hat{\omega}_1(\hat{\goth s},\gamma(\alpha)\hat{\goth s}') = -\hat{\omega}_1(\gamma(\alpha)\hat{\goth s},\hat{\goth s}')\,,
\end{equation}
for all $\alpha\in TM_1$ and ${\goth s},{\goth s}'\in\bbS$ (either ``ordinary'', or ``odd'' spinors).
Indeed it is possible to obtain non-trivial Dirac actions using \(\hat{\omega}_1\), see Lemma~\ref{indep} below.

In super algebra, the vector bundle \(\hat{\bbS}\) is constructed by extending the scalars to some anti-commuting ring.
But, how can~\eqref{oddspinor} be understood within the classical setting, thereby avoiding the notion of super algebra and super manifolds?
A kind of ``cheap'' way in doing so is to use to assume that also the target manifold $M_2$ has a symplectic structure $\omega_2$ and to use a twisted spinor bundle.
More precisely, we replace $\varphi^\ast g_2$ by $\varphi^\ast\omega_2$ to obtain on the twisted spinor bundle
\begin{equation}
\label{twistspin2}
	\pi_{\mbox{\tiny E}}\colon E \coloneqq \bbS^\ast\ot\varphi^\ast TM_2 \longrightarrow M_1\\
\end{equation}
the symmetric inner product
\begin{equation}
		{\big(\phi,\psi\big)}_{\!\mbox{\tiny E}} = \epsilon^{kl}\, \omega_2\big(\phi_k,\psi_l\big) = {\big(\psi,\phi\big)}_{\!\mbox{\tiny E}}\,.
\end{equation}
Here, the symplectic form $\omega^\ast_1$ on the dual spinor bundle is defined in terms of the dual of the Riemannian volume form on $M_1$.
Furthermore, we set $\epsilon^{kl}\equiv\omega_1^*(e^k,e^l)$ for any symplectic orthonormal basis $\mathfrak{s}^1,\mathfrak{s}^2\in\bbS^*$.

As a matter of notation we use the shorthand
\begin{equation}
	\omega_2\big(\phi_k,\psi_l\big) \equiv \phi_k\cdot\psi_l\,,
\end{equation}
such that formally the coefficients of the spinor fields $\psi=\mathfrak{s}^k\otimes\psi_k$ and $\phi=\mathfrak{s}^l\otimes\phi_l$ anti-commute, i.e.
\begin{equation}
\label{oddness1}
	\phi_k\cdot\psi_l = -\psi_k\cdot\phi_l\,.
\end{equation}

One then proves the following
\begin{lemma}\label{indep}
	For every Clifford connection on the twisted spinor bundle \(E\) (see Equation~\eqref{twistspin2}) the Dirac action~\eqref{da} is a real-valued non-trivial functional.
	Furthermore, the Dirac action~\eqref{da} does not depend on the metric connection used on the spinor bundle.
\end{lemma}
\begin{proof}
	The prove of the statement relies crucially on~\eqref{symrep}.
	Consequently, for $k = 1,2$ one gets
	\begin{equation}
		{\big(\psi,\gamma^k\gamma^1\gamma^2\psi\big)}_{\!\mbox{\tiny E}} = 0\,.
	\end{equation}
	Hence,
	\begin{equation}
		{\big(\psi,\dda\psi\big)}_{\!\mbox{\tiny E}} = {\big(\psi,\ddd\psi\big)}_{\!\mbox{\tiny E}}\,.
	\end{equation}
\end{proof}
Consequently, the action functionals of Dirac harmonic maps~\eqref{adhm} and super harmonic maps~\eqref{sadhm1} can be realized also for the Clifford algebra \(Cl_{2,0}\) using the twisted spinor bundle \(E\) instead of \(\mathcal{E}\) in the Dirac action.
We put emphasize, that already in this setting both functionals have all desired symmetry properties besides supersymmetry.
In particular, the Dirac action does not vanish and the functional~\eqref{sadhm1} is still real-valued as opposed to its supersymmetric analogue.
Notice, however, that the realization of \(E\) provides a severe restriction on the target manifolds of the map $\varphi$.
For example, it does not work in the most simple case $(M_2,g_2)=\rr$.

We now want to check whether the action functional of Dirac harmonic map is invariant to the first order under the supersymmetry transformations
\begin{equation}
\label{susytraf2}
	\begin{split}
		\delta_{\goth s}\varphi &= \psi({\goth s}) = \omega_{\mbox{\tiny$\bbS^\ast$}}(\psi,{\tilde{\goth s}})\in {\goth S}ec(M_1,\varphi^\ast TM_2)\,,\\
		\delta_{\goth s}\psi &= \delta_\gamma\big(d\varphi\big){\tilde{\goth s}}\in{\goth S}ec(M_1,E)\,.
	\end{split}
\end{equation}
Here, ${\goth s}\in\Gamma(M_1,\bbS)$ denotes an arbitrary spinor field and ${\tilde{\goth s}}\in\Gamma(M_1,\bbS^\ast)$ its symplectic dual that is defined by ${\tilde{\goth s}} \coloneqq \omega_1({\goth s},\cdot)$.

Notice that our convention for the symplectic dual is the following: for given symplectic orthonormal basis ${\goth s}_1,{\goth s}_2\in\bbS$, with dual basis ${\goth s}^1,{\goth s}^2\in\bbS^\ast$, the symplectic dual reads: ${\tilde{\goth s}}_k \coloneqq \omega_1({\goth s}_k,\cdot) = \epsilon_{kj}{\goth s}^j\in\bbS^\ast$.
Similarly, we define ${\tilde{\goth s}}^l \coloneqq \omega^\ast_1(\cdot,{\goth s}^l) = -\epsilon^{li}{\goth s}_i\in\bbS$, such that for all $k,l=1,2$ we have ${\tilde{\goth s}}^l({\tilde{\goth s}}_k) \coloneqq {\tilde{\goth s}}_k({\tilde{\goth s}}^l) = -{\goth s}^l({\goth s}_k) = -\delta^l_{\;k}$.

To simplify the discussion, we restrict ourselves to the simple case of a trivial spinor bundle over \(M_1=\cc\).
Although the integral~\eqref{adhm} is not well defined in this case, the variation of the integrand is of course.
\begin{proposition}
	Let ${\goth s}_0$ be a constant spinor field.
	The variation of the functional of Dirac harmonic maps~\eqref{adhm} that is defined to first order by the supersymmetry transformations~\eqref{susytraf2} is given by
	\begin{equation}
	\label{varform1}
		\begin{split}
			\delta_{\mathfrak{s}_0}\|d\varphi\|^2 &= \delta^{kl}\, g_2\big(d\delta_{\mathfrak{s}_0}\varphi(e_k),d\varphi(e_l)\big) = -{\big<\psi,{\tilde{\goth s}}_0\!\ot\!\triangle\varphi\big>}_{\!\mbox{\tiny E}} + \Div J_{\mbox{\tiny$\varphi$}}\,,\\
			\delta_{\mathfrak{s}_0} {\big(\psi,\ddd\psi\big)}_{\!\mbox{\tiny$E$}} &=\epsilon^{kl}\,\delta_{\mathfrak{s}_0}\psi_k\!\cdot\ddd\psi_l = -{\big(\psi,{\tilde{\goth s}}_0\!\ot\!\triangle\varphi\big)}_{\!\mbox{\tiny E}} + \Div J_{\mbox{\tiny$\psi$}}\,,
		\end{split}
	\end{equation}
	with the local tangent vector fields $J_{\mbox{\tiny$\varphi$}}, J_{\mbox{\tiny$\psi$}}\in\Gamma(M_1,TM_1)$ being given by
	\begin{equation}
		\begin{split}
			2J_{\mbox{\tiny$\varphi$}} &\coloneqq {\big<\psi,{\tilde{\goth s}}_0\!\ot\!\grad\varphi\big>}_{\!\mbox{\tiny E}}\,,\\
			2J_{\mbox{\tiny$\psi$}} &\coloneqq -2{\big(\psi,\Theta_x^\sharp\gamma(d\varphi){\goth s}_0\big)}_{\!\mbox{\tiny E}}\,.
		\end{split}
	\end{equation}
\end{proposition}
\begin{proof}
	The statement is shown to hold true by a straightforward calculation.
\end{proof}

The statement clearly demonstrates that, using \({\left(\cdot, \cdot\right)}_E\) instead of \({\left\langle\cdot, \cdot\right\rangle}_E\) for the Dirac-term, the functional of Dirac-harmonic maps is not invariant with respect to the supersymmetry transformations~\eqref{susytraf2} even for constant variational spinor fields.
The detailed calculation demonstrates the failure in vanishing of the variation has its origin in the different inner products used on the twisted spinor bundle for the “bosonic” and “fermionic” parts of the functional of Dirac harmonic maps.
Of course, this holds true also for the case of non-trivial spinor bundle and the more general functional~\eqref{sadhm1} with supersymmetry transformations~\eqref{susytraf1a}--\eqref{susytraf1d}.

On the other hand, the only way to circumvent a trivial Dirac action on Riemann surfaces (for real spinors) consists in the usage of the symplectic form on the (real) spinor bundle instead of the usual metric.
Again, this guarantees (symplectic) skew-symmetry of the Clifford action necessary for a symmetric Dirac operator.

We thus proceed with a geometrical construction very similar to what has been used in order to derive certain functionals, like~\eqref{dhym}, from Dirac operators of simple type (c.f.~\cite{T-EHACCFGF}).
The symmetry property~\eqref{oddspinor} alone has not proven sufficient to reproduce super symmetry.
Therefore we specify more properties of odd spinors and give another construction of \({\left(\cdot, \cdot\right)}_E\).
That is, for given Grassmann algebra $\Lambda$ we consider the \emph{Grassmann extension} of the twisted spinor bundle~\eqref{twistspin2}
\begin{equation}
\label{twistspin3}
	E_{\mbox{\tiny$\Lambda$}} \coloneqq \bbS^\ast\!\ot\!\varphi^\ast TM_2\!\ot\!\Lambda \longrightarrow M_1\\
\end{equation}
Similarly, one may replace the ordinary tangent bundle by its Grassmann extension $TM_1\ot\Lambda\rightarrow M_1$, etc..

Clearly, the Grassmann extension of any vector bundle contains the latter as a distinguished sub-vector bundle.
For instance,~\eqref{twistspin3} contains~\eqref{twistspin2} because of the canonical embedding
\begin{equation}
	\begin{split}
		E&\hookrightarrow E_{\mbox{\tiny$\Lambda$}}\\
		{\goth z}&\mapsto{\goth z}\ot 1\,.
	\end{split}
\end{equation}
Of course, this holds true similarly for every Grassmann extension.
Accordingly, by a slight abuse of notation, we do not make a distinction between a vector bundle and its Grassmann extension.
That is, for given $\Lambda$ every (real) vector bundle is considered to be contained into its Grassmann extension.
By construction all vector bundles are naturally $\zz_2-$graded. By an ``even/odd'' section we thus mean section restricted to $\Lambda^{\!\pm}$.

We consider the following graded inner (fiber) product with values in the $\zz_2-$graded ring $\Lambda$:
\begin{equation}
\label{gradprod}
	\begin{split}
		{\bigl(\cdot,\cdot\bigr)}_{\!\mbox{\tiny$E$}}\colon E\times_M E&\longrightarrow\Lambda\\
		({\goth z}_1\ot\lambda_1,{\goth z}_2\ot\lambda_2) &\mapsto \langle{\goth z}_1,{\goth z}_2\rangle_{\!\mbox{\tiny$E$}}\,\lambda_1\wedge\lambda_2\,.
	\end{split}
\end{equation}
Here, for all homogeneous elements ${\goth z}_k = {\goth s}^\ast_k\ot y_k\in E\;(k=1,2)$ we put
\begin{equation}
	\langle{\goth s}^\ast_1\ot y_1,{\goth s}^\ast_2\ot y_2\rangle_{\!\mbox{\tiny$E$}} \coloneqq \omega^\ast_1({\goth s}^\ast_1,{\goth s}^\ast_2)\,\varphi^\ast\!g_2(y_1,y_2)\in\rr\,.
\end{equation}
It follows that for all elements of definite parity
\begin{equation}
	\big({\goth z},{\goth w}\big)_{\!\mbox{\tiny$E$}} = -(-1)^{|{\goth z}||{\goth w}|}\big({\goth w},{\goth z}\big)_{\!\mbox{\tiny$E$}}\,.
\end{equation}

\begin{lemma}
	When restricted to odd sections, the twisted spin-Dirac operator
	$\ddd = \delta_\gamma(\nabla^{\bbS\ot\varphi^\ast TM_2})$ is symmetric with respect to the
	scalar product
	\begin{equation}
	\begin{split}
		\big(\!\!\big(\phi,\psi\big)\!\!\big) &\coloneqq \int_{M_1}\big(\phi,\psi\big)_{\!\mbox{\tiny E}}\,dvol(g_1)\\
			&= \big(\!\!\big(\psi,\phi\big)\!\!\big)\in\Lambda^{\!+}\,.
		\end{split}
	\end{equation}
\end{lemma}
\begin{proof}
	The proof is straightforward and makes use of $\ddd$ being a quantized Clifford connection.
	Hence,
	\begin{equation}
		\nabla^{\mbox{\tiny$T^\ast\!M_1\!\ot\!\bbS\!\ot\!\varphi^{\!\ast}\!TM_2$}}\Theta = 0\,.
	\end{equation}
	Furthermore, the Clifford action is skew-symplectic. Therefore,
	\begin{equation}
		\big(\Theta^\sharp\phi,\psi\big)_{\!\mbox{\tiny E}} + \big(\phi,\Theta^\sharp\psi\big)_{\!\mbox{\tiny E}} = 0\,.
	\end{equation}
	Altogether the calculation yields
	\begin{equation}
		\Div J = -\big(\ddd\phi,\psi\big)_{\!\mbox{\tiny E}} + \big(\phi,\ddd\psi\big)_{\!\mbox{\tiny E}}\,,
	\end{equation}
	with the (even) tangent vector field $J \in\Gamma(M_1,TM_1) $ being given by
	\begin{equation}
		J \coloneqq 2\,\big(\phi,\Theta^\sharp\psi\big)\,.
	\end{equation}
\end{proof}

As a consequence, the Dirac action for odd (twisted) Majorana spinors
\begin{equation}
\big(\!\!\big(\psi,\ddd\psi\big)\!\!\big) = \int_{M_1}\big(\psi,\ddd\psi\big)_{\!\mbox{\tiny E}}\,dvol(g_1) \in\Lambda^{\!+}
\end{equation}
is non-trivial as opposed to the case of even Majorana spinors.
Similar to Lemma~\eqref{indep} one may prove that the Dirac action is actually independent of the metric connection used on the spinor bundle.

\begin{proposition}\label{propSDH}
	Assume that the odd variational spinor field ${\goth s}\in\Gamma_{\mbox{\tiny hol}}(M_1,\bbS^-)$ is holomorphic
	and $(M_2,g_2)$ be flat. The functional
	\begin{equation}
	\label{totdiractionres8}
		\mathcal{A}_{\mbox{\tiny SDH}}(\varphi,\psi) \coloneqq \!\int_{M_1}\!\!\Big(\|d\varphi\|^2 + {\big(\psi,\dda\psi\big)}_{\!\mbox{\tiny E}}\Big)dvol(g_1) \in \Lambda^{\!+}
	\end{equation}
	is stationary with respect to the variation determined to first order by the supersymmetry transformations~\eqref{susytraf2}.
\end{proposition}
\begin{proof}
	As the integrand in Equation~\eqref{totdiractionres8} is conformally invariant, we are allowed to work in coordinates \(x^1, x^2\) on \(M_1\) such that \(g_1={\left(dx^1\right)}^2 + {\left(dx^2\right)}^2\).
	Furthermore, as the integrand does not depend on \(A\), we will work with \(A=0\).
	It follows that in the specific coordinates chosen, the covariant derivative can be expressed as \(d_\nabla\mathfrak{s} = d\mathfrak{s}\).
	For arbitrary but fixed $t\in M_1$ we denote the (oriented) orthonormal frame \(\partial_{x^1}, \partial_{x^2}\) by $e_1,e_2\in T_t M_1$ and consider the following $\Lambda-$valued smooth mappings:
	\begin{equation}
		\begin{split}
			\|d\varphi_t\|^2\colon  (-\epsilon,\epsilon) &\longrightarrow \Lambda^{\!+}\\
			s &\mapsto \delta^{ij}\varphi^\ast\!g_2(d\Phi(s,t)e_i,d\Phi(s,t)e_j)\,,\\[.2cm]
			{\big(\psi_t,\dda\psi_t\big)}_{\!\mbox{\tiny E}}\colon (-\epsilon,\epsilon) &\longrightarrow \Lambda^{\!+}\\
			s &\mapsto \omega^\ast_1({\goth s}^i,\gamma^k{\goth s}^j)\, \varphi^\ast\!g_2(\Psi_i(s,t),d\Psi_j(s,t)e_k)\,,
		\end{split}
	\end{equation}
	whereby to first order for all $v\in T_t M_1$
	\begin{equation}
		\begin{split}
			d\Phi(s,t)v &\coloneqq d\varphi(t)v + sd\psi({\goth s})(t)v + o(s)\in T_{\varphi(t)}M_2\,,\\[,2cm]
			\Psi(s,t) &\coloneqq \psi(t) + s\gamma(d\varphi(t)){\tilde{\goth s}}(t) + o(s)\in E_t^-\,.
		\end{split}
	\end{equation}

	Accordingly, one obtains
	\begin{equation}
		\begin{split}
			\frac{d\|d\varphi_t\|^2}{ds}(s)|_{s=0} &= -2\big(\psi,\triangle\varphi(t)\ot{\tilde{\goth s}}(t)\big)_{\!\mbox{\tiny E}} + \Div{\goth J}_ {\mbox{\tiny$\varphi$}}(t)\\[.2cm]
			\frac{d(\psi_t,\dda\psi_t)}{ds}(s)|_{s=0} &= 2\big(\psi,\triangle\varphi(t)\ot{\tilde{\goth s}}(t)\big)_{\!\mbox{\tiny E}} - 2\big(\psi(t),\gamma^k\gamma(d\varphi(t))d{\goth s}(t)e_k(x)\big)_{\!\mbox{\tiny E}}\\
				&\,+ 2\,\Div{\goth J}_{\mbox{\tiny$\psi$}}(t)\,,
		\end{split}
	\end{equation}
	with the even vector fields ${\goth J}_{\mbox{\tiny$\phi$}},\,{\goth J}_{\mbox{\tiny$\psi$}}\in\Gamma(M_1,TM_1)$ being given by
	\begin{equation}
		\begin{split}
			{\goth J}_{\mbox{\tiny$\varphi$}} &\coloneqq 2\big(\psi,\grad\varphi\!\ot\!{\goth s}\big)_{\!\mbox{\tiny E}}\,,\\
			{\goth J}_{\mbox{\tiny$\psi$}} &\coloneqq 4\big(\psi,\Theta^\sharp\gamma(d\varphi){\goth s}\big)_{\!\mbox{\tiny E}}\,.
		\end{split}
	\end{equation}

	Therefore,
	\begin{equation}
		\begin{split}
			\frac{d\|d\varphi_t\|^2}{ds}(s)|_{s=0} +\frac{d(\psi_t,\dda\psi_t)}{ds}(s)|_{s=0} &= -2\big(\psi(t),\gamma^k\gamma(d\varphi(t))d{\goth s}(t)e_k(t)\big)_{\!\mbox{\tiny E}}\\
			&\quad+ 2\,\Div{\goth J}_{\mbox{\tiny susy}}\,,
		\end{split}
	\end{equation}
	with the ``susy-current'' ${\goth J}_{\mbox{\tiny susy}}\in\Gamma(M_1,TM_1)$ reading
	\begin{equation}
		\begin{split}
			{\goth J}_{\mbox{\tiny susy}} &\coloneqq
			\big(\psi,\grad\varphi\!\ot\!{\goth s}\big)_{\!\mbox{\tiny E}} +
			4\big(\psi,\Theta^\sharp\gamma(d\varphi){\goth s}\big)_{\!\mbox{\tiny E}}\,.
		\end{split}
	\end{equation}

	Since
	\begin{equation}
		2\gamma^k\gamma(d\varphi)d{\goth s}(e_k) = {\goth q}(d{\goth s})\grad\varphi\,,
	\end{equation}
	the statement follows from~\eqref{eq:PQProjectors}.
	Indeed, with \(\theta=dz\), for the holomorphic coordinate \(z=x^1+\ic x^2\) and \(\mathfrak{s}=u\mathfrak{w}\), it holds that
	\begin{equation}
		0=\mathfrak{q}\left(d\mathfrak{s}\right) = \mathfrak{q}\left(\left(\partial_{z}u\right)\theta\otimes\mathfrak{w} + \left(\partial_{\overline{z}}u\right)\overline{\theta}\otimes\mathfrak{w}\right) =  \left(\partial_{\overline{z}}u\right)\overline{\theta}\otimes\mathfrak{w}
	\end{equation}
	Notice, that even though the last equation is obtained using particular local coordinates, due to conformal invariance it glues to the global condition, that \(\mathfrak{s}\) be holomorphic.
\end{proof}

The statement of the Proposition~\eqref{propSDH} can be generalized, actually.
By a similar but (much) more involved calculation (analogous to the corresponding calculation presented in all details in~\cite{EK-DR}, see also~\cite{BdVH-LSRIASS}) one may finally prove the following
\begin{proposition}
	Let, again, $(M_2,g_2)$ be flat. Also, let ${\goth s}\in\Gamma(M_1,\bbS^-)$ be an arbitrary odd spinor field and ${\tilde{\goth s}}\in\Gamma(M_1,\bbS^\ast)$ its symplectic dual.
	The functional
	\begin{gather}
	\label{SRS}
		\mathcal{A}_{\mbox{\tiny SRS}}(\varphi,\psi,\chi,g_1) \coloneqq\\
		\!\int_{M_1}\!\!\Big(\|d\varphi\|^2 + \big(\psi,\dda\psi\big)_{\!\mbox{\tiny E}} + \;\big(\chi,{\goth q}(\chi)\big)_{\mbox{\tiny$T^\ast\!M\!\ot\!\bbS$}}\big(\psi,\psi\big)_{\!\mbox{\tiny E}} + 4\,\big({\goth q}(\chi)(\grad\varphi),\psi\big)_{\!\mbox{\tiny E}}\Big)dvol(g_1)\,,
	\end{gather}
	with $\chi\in\Omega^1(M_1,\bbS^-)$, is stationary to first order with respect to
	\begin{align}
	\label{susytraf2a}
		\delta_{\goth s}\varphi &\coloneqq \psi({\goth s})\in\Gamma(M_1,\varphi^\ast TM_2)\,,\\
	\label{susytraf2b}
		\quad \delta_{\goth s}\psi &\coloneqq \delta_\gamma\big(d\varphi - ev(\psi,\chi)\big){\tilde{\goth s}}\in \Gamma(M_1,E^-)\,,\\
	\label{susytraf2c}
		\delta_{\goth s} e_k &\coloneqq -4\omega^\ast_1(\delta_{\mbox{\tiny$\Theta^\sharp$}}{\goth s},\chi(e_k))\in \Gamma(M_1,TM_1)\qquad(k=1,2)\,,\\
	\label{susytraf2d}
		\delta_{\goth s}\chi &\coloneqq d_{\mbox{\tiny A}}{\goth s}\in\Omega^1(M_1,\bbS^-\big)\,,
	\end{align}
	if and only if the torsion factorizes as
	\begin{equation}
		\begin{split}
			A &= \omega_1(\delta_{\mbox{\tiny$\gamma$}}(\chi),\chi)\,.
		\end{split}
	\end{equation}
\end{proposition}

The realization of anti-commutative spinors~\eqref{oddspinor} by Grassmann extension of ordinary spinors thus allows a ``semi-classical'' interpretation of the functional of super harmonic maps and its symmetries, including supersymmetry.
The prize to be paid is to deal with functionals taking values in a non-commutative ring instead of being real-valued.
This is rather close, indeed, to super geometry.
Indeed, one can show that the action functional \(\mathcal{A}_{SRS}\) from Equation~\eqref{SRS} is equivalent to the action functional \(\mathcal{I}_{SDH}\)-studied in~\cite{JKT-SRSMG} under the assumption of a trivial family of super Riemann surfaces.
It seems unavoidable to use the full language of super geometry, in particular ringed spaces involving anti commutative variables, and reminiscent to the supersymmetry transformations, which seem to have a clear geometrical interpretation only within the realm of super geometry.

In contrast, the functional~\eqref{sadhm1} and the factorization~\eqref{torfac} of the torsion was shown to be already well-motivated by pure symmetry considerations also within the realm of ordinary differential geometry.

\printbibliography

\end{document}